\documentclass{amsart}
\usepackage{amsmath,amssymb}
\usepackage[dvips]{graphics}
\usepackage[all]{xy}
\newtheorem{Theorem}{Theorem}[section]
\newtheorem{Lemma}[Theorem]{Lemma}
\newtheorem{Proposition}[Theorem]{Proposition}
\newtheorem{Corollary}[Theorem]{Corollary}

\theoremstyle{definition}

\newtheorem{Example}[Theorem]{Example}

\theoremstyle{remark}
\newtheorem{Remark}[Theorem]{Remark}

\def\labelenumi{\rm ({\makebox[0.8em][c]{\theenumi}})}
\def\theenumi{\arabic{enumi}}
\def\labelenumii{\rm {\makebox[0.8em][c]{(\theenumii)}}}
\def\theenumii{\roman{enumii}}

\renewcommand{\thefigure}
{\arabic{section}.\arabic{figure}}
\makeatletter
\@addtoreset{figure}{section}
\def\@thmcountersep{-}
\makeatother

\numberwithin{equation}{section}

\begin{document} 

\title[$\triangle Y$-exchanges and the Conway-Gordon theorems]{$\triangle Y$-exchanges and the Conway-Gordon theorems}

\author{Ryo Nikkuni}
\address{Department of Mathematics, School of Arts and Sciences, Tokyo Woman's Christian University, 2-6-1 Zempukuji, Suginami-ku, Tokyo 167-8585, Japan}
\email{nick@lab.twcu.ac.jp}
\thanks{The first author was partially supported by Grant-in-Aid for Young Scientists (B) (No. 21740046), Japan Society for the Promotion of Science.}

\author{Kouki Taniyama}
\address{Department of Mathematics, School of Education, Waseda University, Nishi-Waseda 1-6-1, Shinjuku-ku, Tokyo, 169-8050, Japan}
\email{taniyama@waseda.jp}
\thanks{The second author was partially supported by Grant-in-Aid for Scientific Research (C) (No. 21540099), Japan Society for the Promotion of Science.}


\date{}

\dedicatory{Dedicated to Professor Shin'ichi Suzuki for his 70th birthday}


\begin{abstract}
Conway-Gordon proved that for every spatial complete graph on $6$ vertices, the sum of the linking numbers over all of the constituent $2$-component links is congruent to $1$ modulo $2$, and for every spatial complete graph on $7$ vertices, the sum of the Arf invariants over all of the Hamiltonian knots is also congruent to $1$ modulo $2$. In this paper, we give a Conway-Gordon type theorem for any graph which is obtained from the complete graph on $6$ or $7$ vertices by a finite sequence of $\triangle Y$-exchanges. 
\end{abstract}

\maketitle

\section{Introduction} 

Throughout this paper we work in the piecewise linear category. Let $f$ be an embedding of a finite graph $G$ into the $3$-sphere. Then $f$ is called a {\it spatial embedding} of $G$ and $f(G)$ is called a {\it spatial graph}. We denote the set of all spatial embeddings of $G$ by ${\rm SE}(G)$. We call a subgraph of $G$ which is homeomorphic to a circle a {\it cycle} of $G$, and a cycle of $G$ which contains exactly $k$ edges a {\it $k$-cycle} of $G$. For a positive integer $n$, $\Gamma^{(n)}(G)$ denotes the set of all cycles of $G$ if $n=1$ and the set of all unions of mutually disjoint $n$ cycles of $G$ if $n\ge 2$. We denote the union of $\Gamma^{(n)}(G)$ over all positive integer $n$ by $\bar{\Gamma}(G)$. In the case of $n=1$, we denote $\Gamma^{(1)}(G)$ by $\Gamma(G)$ simply, and denote the subset of $\Gamma(G)$ consisting of all $k$-cycles of $G$ by $\Gamma_{k}(G)$. For an element $\gamma$ in $\Gamma^{(n)}(G)$ and an element $f$ in ${\rm SE}(G)$, $f(\gamma)$ is none other than a knot in $f(G)$ if $n=1$ and an $n$-component link in $f(G)$ if $n\ge 2$. 

Let $K_{n}$ be the {\it complete graph} on $n$ vertices, namely the simple graph consisting of $n$ vertices in which every pair of distinct vertices is connected by exactly one edge. For spatial embeddings of $K_{6}$ and $K_{7}$, we recall the following, which are called the Conway-Gordon theorems.

\begin{Theorem}\label{CG} 
{\rm (Conway-Gordon \cite{CG83})}
\begin{enumerate}
\item For any element $f$ in ${\rm SE}(K_{6})$, it follows that 
\begin{eqnarray*}
\sum_{\gamma\in \Gamma^{(2)}(K_{6})}{\rm lk}(f(\gamma))\equiv 1\pmod{2},  
\end{eqnarray*}
where ${\rm lk}$ denotes the {\it linking number} in the $3$-sphere. 

\item For any element $f$ in ${\rm SE}(K_{7})$, it follows that 
\begin{eqnarray*}
\sum_{\gamma\in \Gamma_{7}(K_{7})}{\rm Arf}(f(\gamma))\equiv 1\pmod{2},  
\end{eqnarray*}
where ${\rm Arf}$ denotes the {\it Arf invariant} \cite{Rober65}. 

\end{enumerate}
\end{Theorem}

Theorem \ref{CG} implies that for any element $f$ in ${\rm SE}(K_{6})$, there exists an element $\gamma$ in $\Gamma^{(2)}(K_{6})$ such that ${\rm lk}(f(\gamma))$ is odd, and for any element $f$ in ${\rm SE}(K_{7})$, there exists an element $\gamma$ in $\Gamma_{7}(K_{7})$ such that ${\rm Arf}(f(\gamma))=1$. A graph is said to be {\it intrinsically linked} if for any element $f$ in ${\rm SE}(G)$, there exists an element $\gamma$ in $\Gamma^{(2)}(G)$ such that $f(\gamma)$ is a nonsplittable $2$-component link, and to be {\it intrinsically knotted} if for any element $f$ in ${\rm SE}(G)$, there exists an element $\gamma$ in $\Gamma(G)$ such that $f(\gamma)$ is a nontrivial knot. Theorem \ref{CG} also implies that $K_{6}$ is intrinsically linked and $K_{7}$ is intrinsically knotted. Moreover, we can obtain another intrinsically linked (resp. knotted) graph from $K_{6}$ (resp. $K_{7}$) in the following way. A {\it $\triangle Y$-exchange} is an operation to obtain a new graph $G_{Y}$ from a graph $G_{\triangle}$ by removing all edges of a $3$-cycle $\triangle$ of $G_{\triangle}$ with the edges $uv,vw$ and $wu$, and adding a new vertex $x$ and connecting it to each of the vertices $u,v$ and $w$ as illustrated in Fig. \ref{Delta-Y} (we often denote ${ux}\cup {vx}\cup {wx}$ by $Y$). A {\it $Y \triangle$-exchange} is the reverse of this operation. Throughout this paper, the symbols $G_{\triangle},G_{Y},u,v,w$ and $x$ are used as in the sense of Fig. \ref{Delta-Y}. Motwani-Raghunathan-Saran \cite{MRS88} showed that if $G_{\triangle}$ is intrinsically linked (resp. knotted) then $G_{Y}$ is also intrinsically linked (resp. knotted). Thus any graph which is obtained from $K_{6}$ (resp. $K_{7}$) by a finite sequence of $\triangle Y$-exchanges is intrinsically linked (resp. knotted). The set of all graphs obtained from $K_{6}$ (resp. $K_{7}$) by a finite sequence of $\triangle Y$-exchanges consists of six (resp. fourteen) graphs as illustrated in Fig. \ref{Petersen2} (resp. Fig. \ref{Heawood2}). 

\begin{figure}[htbp]
      \begin{center}
\scalebox{0.45}{\includegraphics*{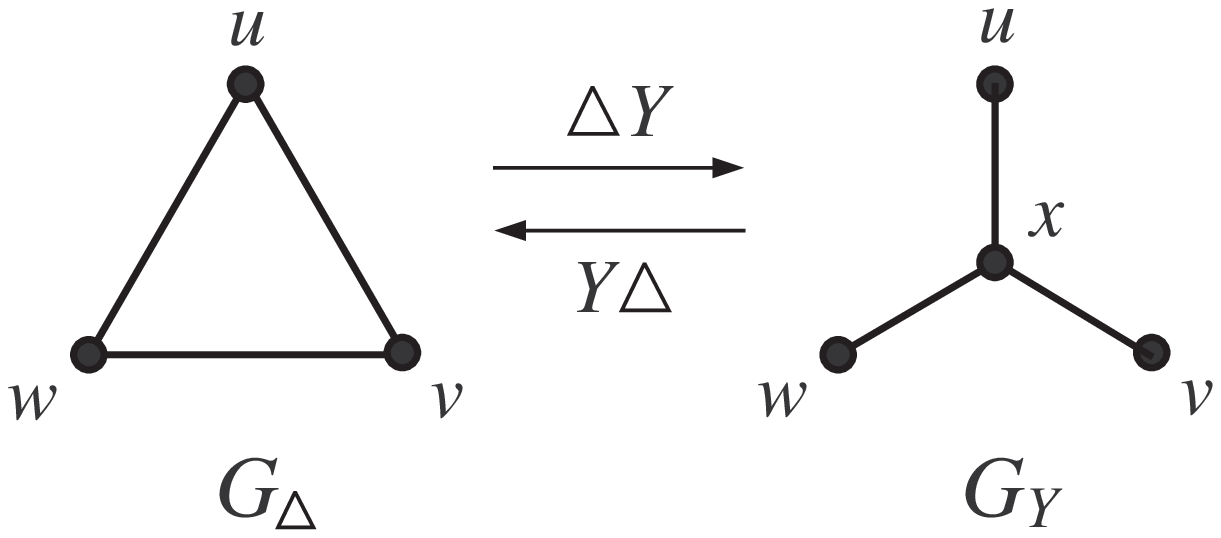}}
      \end{center}
   \caption{}
  \label{Delta-Y}
\end{figure} 
\begin{figure}[htbp]
      \begin{center}
\scalebox{0.5}{\includegraphics*{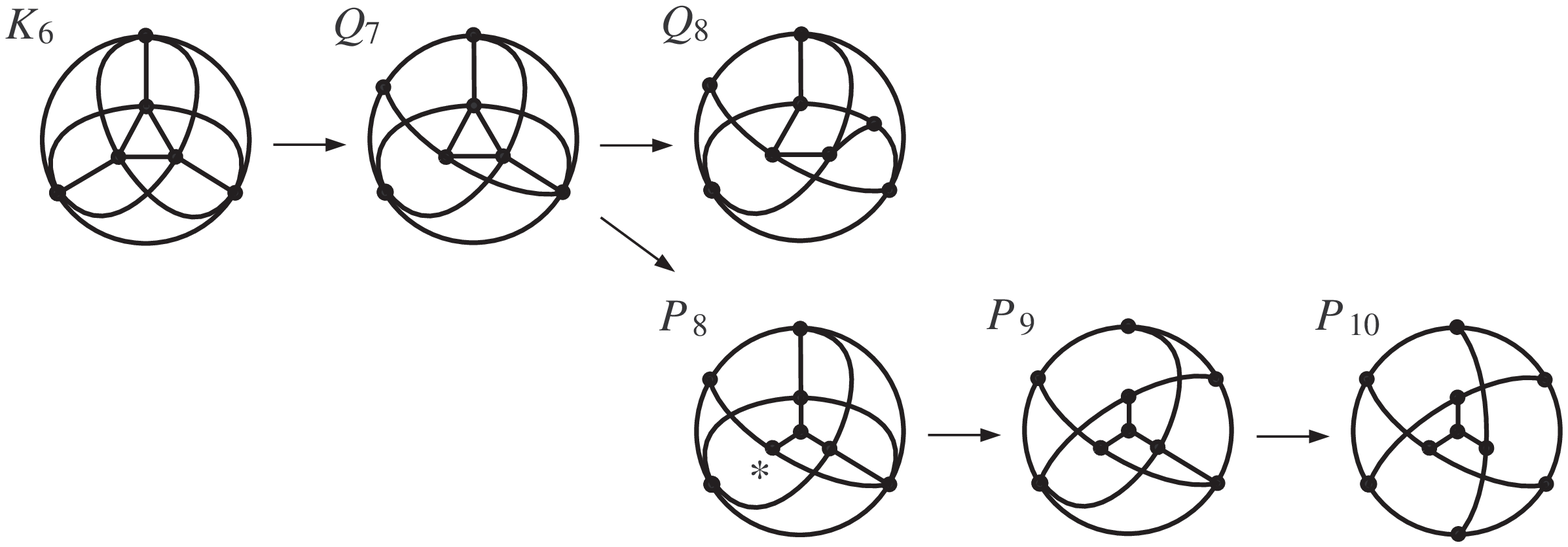}}
      \end{center}
   \caption{}
  \label{Petersen2}
\end{figure} 
\begin{figure}[htbp]
      \begin{center}
\scalebox{0.5}{\includegraphics*{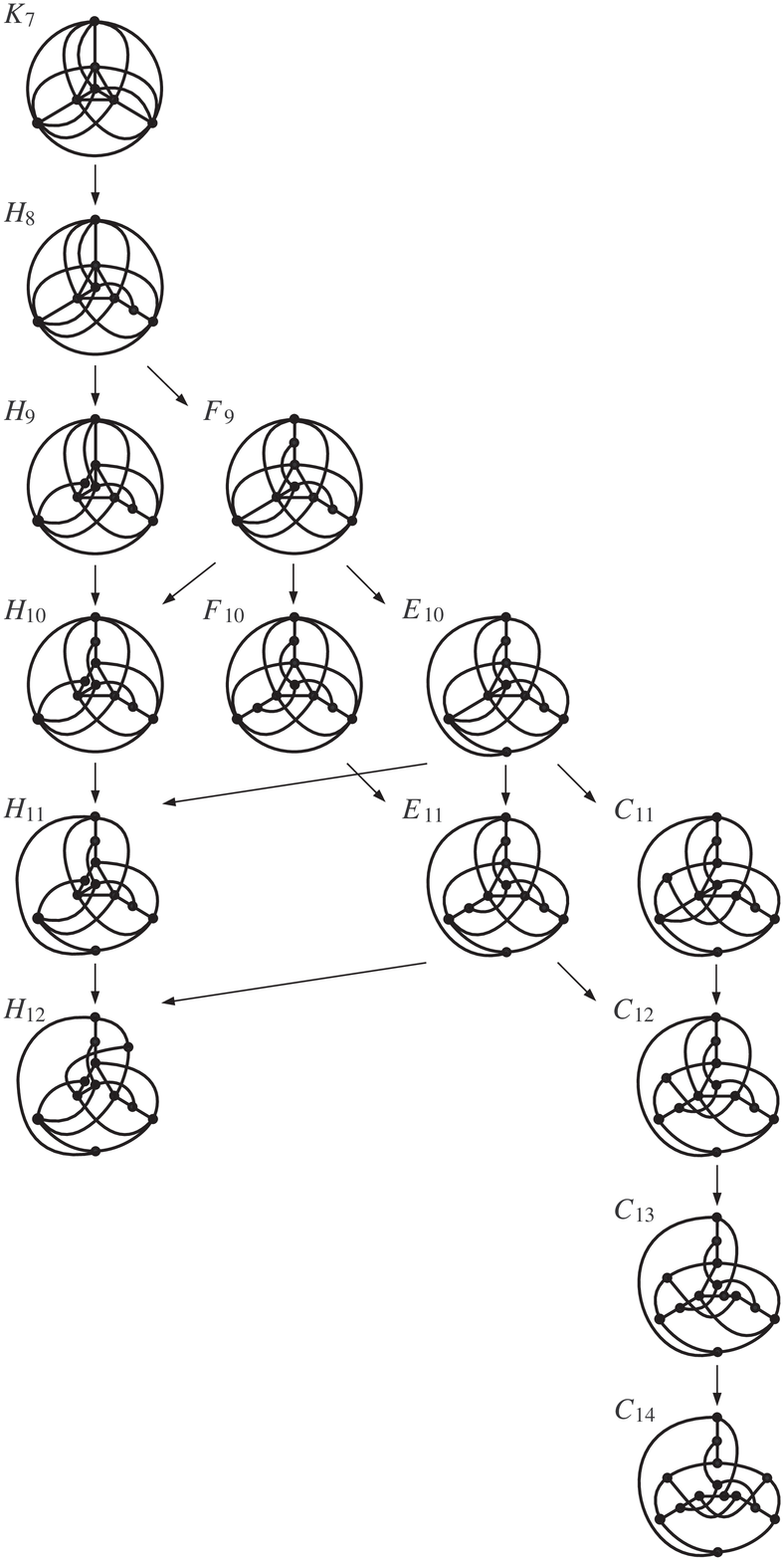}}
      \end{center}
   \caption{}
  \label{Heawood2}
\end{figure} 

Our purpose in this paper is to give a Conway-Gordon type theorem as Theorem \ref{CG} (1) (resp. (2)) for any graph which is obtained from $K_{6}$ (resp. $K_{7}$) by a finite sequence of $\triangle Y$-exchanges. Let $G_{\triangle}$ and $G_{Y}$ be two graphs such that $G_{Y}$ is obtained from $G_{\triangle}$ by a single $\triangle Y$-exchange. We denote the set of all elements in $\bar{\Gamma}(G_{\triangle})$ containing $\triangle$ by $\bar{\Gamma}_{\triangle}(G_{\triangle})$. 
Let $\gamma'$ be an element in $\bar{\Gamma}(G_{\triangle})$ which does not contain $\triangle$. Then there exists an element $\bar{\Phi}(\gamma')$ in $\bar{\Gamma}(G_{Y})$ such that $\gamma'\setminus \triangle=\bar{\Phi}(\gamma')\setminus Y$. It is easy to see that the correspondence from $\gamma'$ to $\bar{\Phi}(\gamma')$ defines a surjective map 
\begin{eqnarray}\label{phi}
\bar{\Phi}=\bar{\Phi}_{G_{\triangle},G_{Y}}:\bar{\Gamma}(G_{\triangle})\setminus \bar{\Gamma}_{\triangle}(G_{\triangle})\longrightarrow \bar{\Gamma}(G_{Y}).
\end{eqnarray}
Let $A$ be an additive group. We say that an $A$-valued unoriented link invariant $\alpha$ is {\it compressible} if $\alpha(L)=0$ for any unoriented link $L$ which have a component $K$ bounding a disk $D$ in the $3$-sphere with $D\cap L = \partial D = K$. Namely $\alpha(L)=0$ if $L$ contains a trivial knot as a split component. In particular $\alpha(L)=0$ when $L$ is a trivial knot. Suppose that for each element $\gamma'$ in $\bar{\Gamma}(G_{\triangle})$, an $A$-valued unoriented link invariant $\alpha_{\gamma'}$ is assigned. Then for each element $\gamma$ in $\bar{\Gamma}(G_{Y})$, we define an $A$-valued unoriented link invariant $\tilde{\alpha}_\gamma$ by
\begin{eqnarray*}
\tilde{\alpha}_\gamma(L)=\sum_{\gamma'\in\bar{\Phi}^{-1}(\gamma)}\alpha_{\gamma'}(L).
\end{eqnarray*}

Then we have the following theorem.

\begin{Theorem}\label{main2} 
Suppose that $\alpha_{\gamma'}$ is compressible for each element $\gamma'$ in $\bar{\Gamma}_\triangle(G_\triangle)$. Suppose that there exists a fixed element $c$ in $A$ such that
\begin{eqnarray*}
\sum_{\gamma'\in\bar{\Gamma}(G_{\triangle})}\alpha_{\gamma'}(g(\gamma'))=c
\end{eqnarray*}
for any element $g$ in ${\rm SE}(G_{\triangle})$. Then we have
\begin{eqnarray*}
\sum_{\gamma\in\bar{\Gamma}(G_{Y})}\tilde{\alpha}_\gamma(f(\gamma))=c
\end{eqnarray*}
for any element $f$ in ${\rm SE}(G_{Y})$.
\end{Theorem}

As an application of Theorem \ref{main2}, we have the following. 

\begin{Theorem}\label{main} 
\begin{enumerate}
\item Let $G$ be a graph which is obtained from $K_{6}$ by a finite sequence of $\triangle Y$-exchanges. Then there exist a map $\omega$ from $\Gamma(G)$ to ${\mathbb Z}$ such that for any element $f$ in ${\rm SE}(G)$, it follows that 
\begin{eqnarray*}
2\sum_{\gamma\in \Gamma(G)}\omega(\gamma)a_{2}(f(\gamma))
=
\sum_{\gamma\in \Gamma^{(2)}(G)}{\rm lk}(f(\gamma))^{2}
-1, 
\end{eqnarray*}
where $a_{i}$ denotes the $i$th coefficient of the {\it Conway polynomial}. 

\item Let $G$ be a graph which is obtained from $K_{7}$ by a finite sequence of $\triangle Y$-exchanges. Then there exists a map $\omega$ from $\bar{\Gamma}(G)$ to ${\mathbb Z}$ such that for any element $f$ in ${\rm SE}(G)$, it follows that 
\begin{eqnarray*}
\sum_{\gamma\in \Gamma(G)}\omega(\gamma)a_{2}(f(\gamma))
=
2\sum_{\gamma\in \Gamma^{(2)}(G)}\omega(\gamma){\rm lk}(f(\gamma))^{2}
-21. 
\end{eqnarray*}
\end{enumerate}
\end{Theorem}

As we will say later in Theorem \ref{N_refine}, Theorem \ref{main} (1) and (2) has been already shown by the first author in the case $G$ is $K_{6}$ and $K_{7}$, respectively \cite{N09b}. Theorem \ref{main} is shown by combining Theorem \ref{main2} and Theorem \ref{N_refine}. 

Note that the square of the linking number is congruent to the linking number modulo two, and the second coefficient of the Conway polynomial of a knot is congruent to the Arf invariant modulo two \cite{Ka83}. Thus by taking the modulo two reduction in Theorem \ref{main}, we have the following corollary. 

\begin{Corollary}\label{main_cor} 
\begin{enumerate}
\item Let $G$ be a graph which is obtained from $K_{6}$ by a finite sequence of $\triangle Y$-exchanges. Then, for any element $f$ in ${\rm SE}(G)$, it follows that 
\begin{eqnarray*}
\sum_{\gamma\in \Gamma^{(2)}(G)}{\rm lk}(f(\gamma))
\equiv 1 \pmod{2}. 
\end{eqnarray*}

\item Let $G$ be a graph which is obtained from $K_{7}$ by a finite sequence of $\triangle Y$-exchanges. Then there exists a map $\omega$ from $\Gamma(G)$ to ${\mathbb Z}_{2}$ such that for any element $f$ in ${\rm SE}(G)$, it follows that 
\begin{eqnarray*}
\sum_{\gamma\in \Gamma(G)}\omega(\gamma)a_{2}(f(\gamma)) \equiv 1\pmod{2}. 
\end{eqnarray*}
In other words, there exists a subset $\Gamma$ of $\Gamma(G)$ such that for any element $f$ in ${\rm SE}(G)$, it follows that 
\begin{eqnarray*}
\sum_{\gamma\in \Gamma}{\rm Arf}(f(\gamma)) \equiv 1 \pmod{2}.
\end{eqnarray*}
\end{enumerate}
\end{Corollary}

Note that Corollary \ref{main_cor} (1) has already pointed out by Sachs \cite{S84}\footnote{For a graph $G$ which is obtained from $K_{6}$ by a finite sequence of $\triangle Y$ and $Y \triangle$-exchanges, Sachs showed that for every spatial embedding of $G$, the sum of certain geometric positive integer-valued invariants over all of the constituent $2$-component links is odd. In fact, this geometric invariant is congruent to the linking number modulo two.} and the second author-Yasuhara \cite{TY01}, but as far as the authors know, Corollary \ref{main_cor} (2) has not been known yet except the case $G$ is $K_{7}$, see also Remark \ref{comments}. 

\begin{Remark}\label{main_rem} 
\begin{enumerate}
\item 
The set of all graphs obtained from $K_{6}$ by a finite sequence of $\triangle Y$ and $Y \triangle$-exchanges is called the {\it Petersen family}. This family consists of six graphs of Fig. \ref{Petersen2} and the complete tripartite graph $K_{3,3,1}$ which cannot be obtained from $K_{6}$ by a finite sequence of $\triangle Y$-exchanges ($K_{3,3,1}$ is obtained from $P_{8}$ by a single $Y \triangle$-exchange at marked $Y$ as illustrated in Fig. \ref{Petersen2}). It is known that $K_{3,3,1}$ is also intrinsically linked \cite{S84} and it follows that 
\begin{eqnarray*}
\sum_{\gamma\in\Gamma^{(2)}(K_{3,3,1})}{\rm lk}(f(\gamma))\equiv 1\pmod{2}
\end{eqnarray*}
for any element $f$ in ${\rm SE}(K_{3,3,1})$ \cite{TY01}. Recently O'Donnol showed in \cite{D10} that there exist a map $\omega$ from $\Gamma(K_{3,3,1})$ to ${\mathbb Z}$ such that 
\begin{eqnarray*}
2\sum_{\gamma\in \Gamma(K_{3,3,1})}\omega(\gamma)a_{2}(f(\gamma))
=
\sum_{\gamma\in \Gamma^{(2)}(K_{3,3,1})}{\rm lk}(f(\gamma))^{2}
-1
\end{eqnarray*}
for any element $f$ in ${\rm SE}(K_{3,3,1})$. Namely, an integral version of the Conway-Gordon type theorem as Theorem \ref{main} (1) holds for any graph in the Petersen family.  

\item 
The set of all graphs which is obtained from $K_{7}$ by a finite sequence of $\triangle Y$ and $Y \triangle$-exchanges is called the {\it Heawood family}. This family consists of fourteen graphs of Fig. \ref{Heawood2} and the other six graphs which cannot be obtained from $K_{7}$ by a finite sequence of $\triangle Y$-exchanges. It is known that a graph in the Heawood family is intrinsically knotted if and only if the graph is obtained from $K_{7}$ by a finite sequence of $\triangle Y$-exchanges \cite{HNTY10}. Namely, an integral version of the Conway-Gordon type theorem as Theorem \ref{main} (2) holds for any graph in the Heawood family which is intrinsically knotted. 
\end{enumerate}
\end{Remark}

In the next section, we prove Theorem \ref{main2}. In section $3$, we prove Theorem \ref{main}.

\section{Proof of Theorem \ref{main2}}

Let $f$ be a spatial embedding of $G_{Y}$ and $D$ a $2$-disk in the $3$-sphere such that $D\cap f(G_{Y})=f(Y)$ and $\partial D \cap f(G_{Y}) = \{f(u),f(v),f(w)\}$. Let $\varphi(f)$ be a spatial embedding of $G_{\triangle}$ such that $\varphi(f)(x)=f(x)$ for $x\in G_{\triangle}\setminus \triangle = G_{Y}\setminus Y$ and $\varphi(f)(G_{\triangle})=\left(f(G_{Y})\setminus f(Y)\right)\cup \partial D$. Thus we obtain a map 
\begin{eqnarray}\label{varphi}
\varphi:{\rm SE}(G_{Y})\longrightarrow {\rm SE}(G_{\triangle}). 
\end{eqnarray}
Then we immediately have the following. 

\begin{Proposition}\label{map} 
Let $f$ be an element in ${\rm SE}(G_{Y})$ and $\gamma$ an element in $\bar{\Gamma}(G_{Y})$. Then, $f(\gamma)$ is ambient isotopic to $\varphi(f)(\gamma')$ for each element $\gamma'$ in the inverse image of $\gamma$ by $\bar{\Phi}$. \hfill $\square$
\end{Proposition}

Suppose that for each element $\gamma'$ in $\bar{\Gamma}(G_{\triangle})$, an $A$-valued unoriented link invariant $\alpha_{\gamma'}$ is assigned. Then we have the following lemma. 

\begin{Lemma}\label{lemma2a} 
If $\alpha_{\gamma'}$ is compressible for any element $\gamma'$ in $\bar{\Gamma}_{\triangle}(G_{\triangle})$, then we have 
\begin{eqnarray*}
\sum_{\gamma\in\bar{\Gamma}(G_{Y})}\tilde{\alpha}_{\gamma}(f(\gamma))
=
\sum_{\gamma'\in\bar{\Gamma}(G_\triangle)}\alpha_{\gamma'}(\varphi(f)(\gamma'))
\end{eqnarray*}
for any element $f$ in ${\rm SE}(G_Y)$. 
\end{Lemma}

\begin{proof}
For an element $\gamma'$ in $\bar{\Gamma}_{\triangle}(G_{\triangle})$, we see that $\varphi(f)(\gamma')$ is the trivial knot if $\gamma'$ belongs to $\Gamma(G_{\triangle})$ and a link containing a trivial knot as a split component if $\gamma'$ belongs to $\bar{\Gamma}(G_{\triangle})\setminus \Gamma(G_{\triangle})$. Since $\alpha_{\gamma'}$ is compressible for any element $\gamma'$ in $\bar{\Gamma}(G_{\triangle})$, we see that 
\begin{eqnarray*}
\sum_{\gamma'\in\bar{\Gamma}(G_\triangle)}\alpha_{\gamma'}(\varphi(f)(\gamma'))
= \sum_{\gamma'\in\bar{\Gamma}(G_{\triangle})\setminus \bar{\Gamma}_{\triangle}(G_{\triangle})}\alpha_{\gamma'}(\varphi(f)(\gamma')).
\end{eqnarray*}
Note that 
\begin{eqnarray*}
\bar{\Gamma}(G_{\triangle})\setminus \bar{\Gamma}_{\triangle}(G_{\triangle})
=\bigcup_{\gamma\in\bar{\Gamma}(G_{Y})}\bar{\Phi}^{-1}(\gamma). 
\end{eqnarray*}
Then, by Proposition \ref{map}, we see that 
\begin{eqnarray*}
\sum_{\gamma'\in\bar{\Gamma}(G_{\triangle})\setminus \bar{\Gamma}_{\triangle}(G_{\triangle})}\alpha_{\gamma'}(\varphi(f)(\gamma'))
&=&\sum_{\gamma\in\bar{\Gamma}(G_Y)}\left(\sum_{\gamma'\in\bar{\Phi}^{-1}(\gamma)}\alpha_{\gamma'}(\varphi(f)(\gamma'))\right)\\
&=&\sum_{\gamma\in\bar{\Gamma}(G_Y)}\left(\sum_{\gamma'\in\bar{\Phi}^{-1}(\gamma)}\alpha_{\gamma'}(f(\gamma))\right)
\\
&=& \sum_{\gamma\in\bar{\Gamma}(G_{Y})}\tilde{\alpha}_{\gamma}(f(\gamma)). 
\end{eqnarray*}
Thus we have the result. 
\end{proof}

\begin{proof}[Proof of Theorem \ref{main2}.] 
Suppose that there exists a fixed element $c$ in $A$ such that
\begin{eqnarray}\label{assume}
\sum_{\gamma'\in\bar{\Gamma}(G_{\triangle})}\alpha_{\gamma'}(g(\gamma'))=c
\end{eqnarray}
for any element $g$ in ${\rm SE}(G_{\triangle})$. Then by Lemma \ref{lemma2a} and (\ref{assume}), we have 
\begin{eqnarray*}
\sum_{\gamma\in\bar{\Gamma}(G_{Y})}\tilde{\alpha}_\gamma(f(\gamma))
=\sum_{\gamma'\in\bar{\Gamma}(G_{\triangle})}\alpha_{\gamma'}(\varphi(f)(\gamma'))
=c
\end{eqnarray*}
for any element $f$ in ${\rm SE}(G_{Y})$. 
\end{proof}

\section{Proof of Theorem \ref{main}}

Let $\gamma$ be an element in $\bar{\Gamma}(G_{Y})$. Then we see that the inverse image of $\gamma$ by $\bar{\Phi}$ contains at most two elements in $\bar{\Gamma}(G_{\triangle})\setminus \bar{\Gamma}_{\triangle}(G_{\triangle})$, see Fig. \ref{not_inj}. Moreover, we also see the following. 

\begin{Proposition}\label{injective} 
Let $\gamma$ be an element in $\bar{\Gamma}(G_{Y})$. Then, the inverse image of $\gamma$ by $\bar{\Phi}$ consists of exactly one element if and only if $\gamma$ contains $u,v,w$ and $x$, or $\gamma$ does not contain $x$. \hfill $\square$
\end{Proposition}

\begin{figure}[htbp]
      \begin{center}
\scalebox{0.45}{\includegraphics*{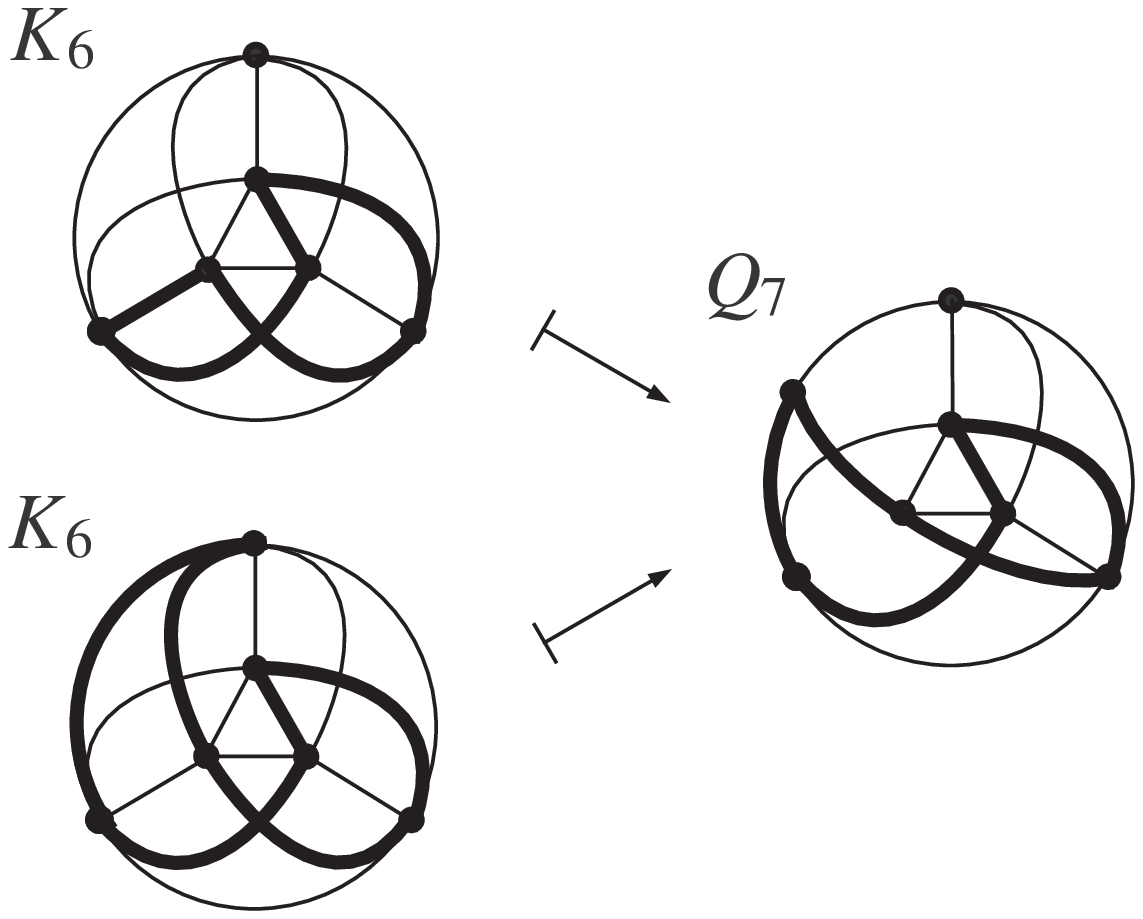}}
      \end{center}
   \caption{}
  \label{not_inj}
\end{figure} 

Note that if $\gamma'$ is an element in $\Gamma^{(n)}(G_{\triangle})\setminus \bar{\Gamma}_{\triangle}(G_{\triangle})$ then $\bar{\Phi}(\gamma')$ is an element in $\Gamma^{(n)}(G_{Y})$. This implies that the restriction map of $\bar{\Phi}$ on $\Gamma^{(n)}(G_{\triangle})\setminus \bar{\Gamma}_{\triangle}(G_{\triangle})$ induces a surjective map 
\begin{eqnarray*}
\Phi^{(n)}=\Phi_{G_{\triangle},G_{Y}}^{(n)}:\Gamma^{(n)}(G_{\triangle})\setminus \bar{\Gamma}_{\triangle}(G_{\triangle})\longrightarrow \Gamma^{(n)}(G_{Y}). 
\end{eqnarray*}
In particular, we denote $\Phi^{(1)}$ by $\Phi$ simply. The surjectivity of $\Phi^{(n)}$ implies that if $\Gamma^{(n)}(G_{\triangle})$ is an empty set then $\Gamma^{(n)}(G_{Y})$ is also an empty set for $n\ge 2$. Since both $\Gamma^{(n)}(K_{6})$ and $\Gamma^{(n)}(K_{7})$ are the empty sets for $n\ge 3$, we have the following. 

\begin{Proposition}\label{empty3} 
Let $G$ be a graph which is obtained from $K_{6}$ or $K_{7}$ by a finite sequence of $\triangle Y$-exchanges. Then $\Gamma^{(n)}(G)$ is an empty set for $n\ge 3$. \hfill $\square$
\end{Proposition}

Now we prove Theorem \ref{main}. First we recall a refinement of the Conway-Gordon theorems which was shown by the first author. 

\begin{Theorem}\label{N_refine} 
{\rm (\cite{N09b})}
\begin{enumerate}
\item For any element $f$ in ${\rm SE}(K_{6})$, it follows that 
\begin{eqnarray*}
2\sum_{\gamma\in \Gamma_{6}(K_{6})}a_{2}(f(\gamma))
-2\sum_{\gamma\in \Gamma_{5}(K_{6})}a_{2}(f(\gamma))
=
\sum_{\gamma\in \Gamma^{(2)}(K_{6})}{\rm lk}(f(\gamma))^{2}
-1. 
\end{eqnarray*}

\item For any element $f$ in ${\rm SE}(K_{7})$, it follows that 
\begin{eqnarray*}
&&7\sum_{\gamma\in \Gamma_{7}(K_{7})}a_{2}(f(\gamma))
-6\sum_{\gamma\in \Gamma_{6}(K_{7})}a_{2}(f(\gamma))
-2\sum_{\gamma\in \Gamma_{5}(K_{7})}a_{2}(f(\gamma))\\
&=&
2\sum_{\gamma\in \Gamma_{4,3}^{(2)}(K_{7})}{\rm lk}(f(\gamma))^{2}
-21, 
\end{eqnarray*}
where $\Gamma_{k,l}^{(2)}(K_{7})$ denotes the set of all unions of two disjoint cycles of $K_{7}$ consisting of a $k$-cycle and an $l$-cycle. \hfill $\square$
\end{enumerate}
\end{Theorem}

Note that Theorem \ref{CG} can be obtained from Theorem \ref{N_refine} by taking the modulo two reduction.

\begin{proof}[Proof of Theorem \ref{main}.] 
\noindent
First we show (1). We define a map $\omega$ from $\bar{\Gamma}(K_{6})$ to ${\mathbb Z}$ by 
\begin{eqnarray*}
\omega(\gamma')=\left\{
       \begin{array}{@{\,}ll}
       1 & \mbox{if $\gamma'\in \Gamma_{6}(K_{6})\cup \Gamma^{(2)}(K_{6})$}\\
       -1 & \mbox{if $\gamma'\in\Gamma_{5}(K_{6})$}\\
       0 & \mbox{otherwise}
       \end{array}
     \right.
\end{eqnarray*}
for an element $\gamma'$ in $\bar{\Gamma}(K_{6})$. Then by Theorem \ref{N_refine} (1), it follows that 
\begin{eqnarray}\label{assump}
2\sum_{\gamma'\in \Gamma(K_{6})}\omega(\gamma')a_{2}(g(\gamma'))
=
\sum_{\gamma'\in \Gamma^{(2)}(K_{6})}\omega(\gamma'){\rm lk}(g(\gamma'))^{2}
-1
\end{eqnarray}
for any element $g$ in ${\rm SE}(K_{6})$. For each element $\gamma'$ in $\bar{\Gamma}(K_{6})$, we define an integer-valued unoriented link invariant $\alpha_{\gamma'}$ of an unoriented link $L$ as follows. If $\gamma'$ is an element in $\Gamma(K_{6})$, then $\alpha_{\gamma'}(L)=2\omega(\gamma')a_{2}(L)$ if $L$ is a knot and $0$ if $L$ is not a knot. If $\gamma'$ is an element in $\Gamma^{(2)}(K_{6})$, then $\alpha_{\gamma'}(L)=-\omega(\gamma')a_{1}(L)^{2}$ if $L$ is a $2$-component link and $0$ if $L$ is not a $2$-component link. Note that $a_{1}(L)^{2}={\rm lk}(L)^{2}$ if $L$ is a $2$-component link. Then by (\ref{assump}), we have 
\begin{eqnarray}\label{const_k6}
\sum_{\gamma'\in \bar{\Gamma}(K_{6})}\alpha_{\gamma'}(g(\gamma'))=-1. 
\end{eqnarray}
Let us consider the graph $Q_{7}$ which is obtained from $K_{6}$ by a single $\triangle Y$-exchange. Note that $\alpha_{\gamma'}$ is compressible for any element $\gamma'$ in $\bar{\Gamma}(K_{6})$. Thus by Theorem \ref{main2} and (\ref{const_k6}), we have 
\begin{eqnarray}\label{const_q7}
\sum_{\gamma\in \bar{\Gamma}(Q_{7})}\tilde{\alpha}_{\gamma}(f(\gamma))=-1 
\end{eqnarray}
for any element $f$ in ${\rm SE}(Q_{7})$. Now we define a map $\tilde{\omega}$ from $\bar{\Gamma}(Q_{7})$ to ${\mathbb Z}$ by 
\begin{eqnarray}\label{c}
\tilde{\omega}(\gamma)
=
\sum_{\gamma'\in \bar{\Phi}^{-1}(\gamma)}\omega(\gamma')
\end{eqnarray}
for an element $\gamma$ in $\bar{\Gamma}(Q_{7})$. Then we have 
\begin{eqnarray}\label{omega}
\tilde{\alpha}_{\gamma}(L)
=
2\sum_{\gamma'\in {\Phi}^{-1}(\gamma)}\omega(\gamma')a_{2}(L)
=
2\tilde{\omega}(\gamma)a_{2}(L)
\end{eqnarray}
for an element $\gamma$ in $\Gamma(Q_{7})$, and 
\begin{eqnarray}\label{xi}
\tilde{\alpha}_{\gamma}(L)
=
-\sum_{\gamma'\in {{\Phi}^{(2)}}^{-1}(\gamma)}\omega(\gamma')a_{1}(L)^{2}
=
-\tilde{\omega}(\gamma)a_{1}(L)^{2}
\end{eqnarray}
for any element $\gamma$ in $\Gamma^{(2)}(Q_{7})$. Recall $\bar{\Gamma}(Q_{7})=\Gamma(Q_{7})\cup \Gamma^{(2)}(Q_{7})$ by Proposition \ref{empty3}. Thus by combining (\ref{const_q7}), (\ref{omega}) and (\ref{xi}), we have 
\begin{eqnarray}\label{a}
2\sum_{\gamma\in \Gamma(Q_{7})}\tilde{\omega}(\gamma)a_{2}(f(\gamma))
-\sum_{\gamma\in \Gamma^{(2)}(Q_{7})}\tilde{\omega}(\gamma){\rm lk}(f(\gamma))^{2}=-1.
\end{eqnarray}
It can be checked directly that each union of mutually disjoint two cycles of a graph in the Petersen family contains all of the vertices of the graph. Thus the map 
\begin{eqnarray*}
\Phi^{(2)}:\Gamma^{(2)}(K_{6})\setminus \bar{\Gamma}_{\triangle}(K_{6})\longrightarrow \Gamma^{(2)}(Q_{7})
\end{eqnarray*}
is bijective by Proposition \ref{injective} and therefore 
\begin{eqnarray}\label{b}
\tilde{\omega}(\gamma)=
\sum_{\gamma'\in {\Phi^{(2)}}^{-1}(\gamma)}\omega(\gamma')
=1
\end{eqnarray}
 for any element $\gamma$ in $\Gamma^{(2)}(Q_{7})$. Thus by (\ref{a}) and (\ref{b}), we have 
\begin{eqnarray*}
2\sum_{\gamma\in \Gamma(Q_{7})}\tilde{\omega}(\gamma)a_{2}(f(\gamma))
=\sum_{\gamma\in \Gamma^{(2)}(Q_{7})}{\rm lk}(f(\gamma))^{2}-1.
\end{eqnarray*}
By repeating this argument, we have the desired conclusion.

Next we show (2). We define a map $\omega$ from $\bar{\Gamma}(K_{7})$ to ${\mathbb Z}$ by 
\begin{eqnarray*}
\omega(\gamma')=\left\{
       \begin{array}{@{\,}ll}
       7 & \mbox{if $\gamma'\in \Gamma_{7}(K_{7})$}\\
       -6 & \mbox{if $\gamma'\in \Gamma_{6}(K_{7})$}\\
       -2 & \mbox{if $\gamma'\in \Gamma_{5}(K_{7})$}\\
       1 & \mbox{if $\gamma'\in \Gamma_{4,3}^{(2)}(K_{7})$}\\
       0 & \mbox{otherwise}
       \end{array}
     \right.
\end{eqnarray*}
for an element $\gamma'$ in $\bar{\Gamma}(K_{7})$. Then by Theorem \ref{N_refine} (2), it follows that 
\begin{eqnarray}\label{assump2}
\sum_{\gamma\in \Gamma(K_{7})}\omega(\gamma)a_{2}(g(\gamma))
=
2\sum_{\gamma\in \Gamma^{(2)}(K_{7})}\omega(\gamma){\rm lk}(g(\gamma))^{2}
-21
\end{eqnarray}
for any element $g$ in ${\rm SE}(K_{7})$. For each element $\gamma'$ in $\bar{\Gamma}(K_{7})$, we define an integer-valued unoriented link invariant $\alpha_{\gamma'}$ of an unoriented link $L$ as follows. If $\gamma'$ is an element in $\Gamma(K_{7})$, then $\alpha_{\gamma'}(L)=\omega(\gamma')a_{2}(L)$ if $L$ is a knot and $0$ if $L$ is not a knot. If $\gamma'$ is an element in $\Gamma^{(2)}(K_{7})$, then $\alpha_{\gamma'}(L)=-2\omega(\gamma')a_{1}(L)^{2}$ if $L$ is a $2$-component link and $0$ if $L$ is not a $2$-component link. Then by (\ref{assump2}), we have 
\begin{eqnarray}\label{const_k7}
\sum_{\gamma'\in \bar{\Gamma}(K_{7})}\alpha_{\gamma'}(g(\gamma'))=-21. 
\end{eqnarray}
Let us consider the graph $H_{8}$ which is obtained from $K_{7}$ by a single $\triangle Y$-exchange. Note also that $\alpha_{\gamma'}$ is compressible for any element $\gamma'$ in $\bar{\Gamma}(K_{7})$. Thus by Theorem \ref{main2} and (\ref{const_k6}), we have 
\begin{eqnarray}\label{const_h8}
\sum_{\gamma\in \bar{\Gamma}(H_{8})}\tilde{\alpha}_{\gamma}(f(\gamma))=-21 
\end{eqnarray}
for any element $f$ in ${\rm SE}(H_{8})$. Now we define a map $\tilde{\omega}$ from $\bar{\Gamma}(H_{8})$ to ${\mathbb Z}$ by 
\begin{eqnarray}\label{c2}
\tilde{\omega}(\gamma)
=
\sum_{\gamma'\in \bar{\Phi}^{-1}(\gamma)}\omega(\gamma')
\end{eqnarray}
for an element $\gamma$ in $\bar{\Gamma}(H_{8})$. Then we can see that 
\begin{eqnarray*}
\sum_{\gamma\in \Gamma(H_{8})}\tilde{\omega}(\gamma)a_{2}(f(\gamma))
= 2\sum_{\gamma\in \Gamma^{(2)}(H_{8})}\tilde{\omega}(\gamma){\rm lk}(f(\gamma))^{2}
-21
\end{eqnarray*}
for any element $f$ in ${\rm SE}(H_{8})$ in the same way as the proof of (1). By repeating this argument, we have the desired conclusion. 
\end{proof}

\begin{Example}\label{ex1} 
Let $\tilde{\omega}$ be the map from $\bar{\Gamma}(Q_{7})$ to ${\mathbb Z}$ as in (\ref{c}). In the following, let us determine $\tilde{\omega}(\gamma)$ for each element $\gamma$ in $\Gamma(Q_{7})$. If $\gamma$ is an element in $\Gamma_{7}(Q_{7})$, then there uniquely exists an element $\gamma'$ in $\Gamma_{6}(K_{6})$ such that $\Phi^{-1}_{K_{6},Q_{7}}(\gamma)=\left\{\gamma'\right\}$. Thus we have $\tilde{\omega}(\gamma)=\omega(\gamma')=1$. If $\gamma$ is an element in $\Gamma_{6}(Q_{7})$, we divide our situation into the following three cases. If $\gamma$ does not contain $x$, then there uniquely exists an element $\gamma'$ in $\Gamma_{6}(K_{6})$ such that $\Phi^{-1}_{K_{6},Q_{7}}(\gamma)=\left\{\gamma'\right\}$. Thus we have $\tilde{\omega}(\gamma)=\omega(\gamma')=1$. If $\gamma$ contains $x$ and does not contain $u,v$ or $w$, then there exists an element $\gamma'_{1}$ in $\Gamma_{5}(K_{6})$ and an element $\gamma'_{2}$ in $\Gamma_{6}(K_{6})$ such that $\Phi^{-1}_{K_{6},Q_{7}}(\gamma)=\left\{\gamma'_{1},\gamma'_{2}\right\}$. Thus we have $\tilde{\omega}(\gamma)=\omega(\gamma'_{1})+\omega(\gamma'_{2})=0$. If $\gamma$ contains $u,v,w$ and $x$, then there uniquely exists an element $\gamma'$ in $\Gamma_{5}(K_{6})$ such that $\Phi^{-1}_{K_{6},Q_{7}}(\gamma)=\left\{\gamma'\right\}$. Thus we have $\tilde{\omega}(\gamma)=\omega(\gamma')=-1$. If $\gamma$ is an element in $\Gamma_{5}(Q_{7})$, we divide our situation into the following two cases. If $\gamma$ does not contain $x$, then there uniquely exists an element $\gamma'$ in $\Gamma_{5}(K_{6})$ such that $\Phi^{-1}_{K_{6},Q_{7}}(\gamma)=\left\{\gamma'\right\}$. Thus we have $\tilde{\omega}(\gamma)=\omega(\gamma')=-1$. If $\gamma$ contains $x$, then note that $\gamma$ does not contain $u,v$ or $w$. Then there exists an element $\gamma'_{1}$ in $\Gamma_{4}(K_{6})$ and an element $\gamma'_{2}$ in $\Gamma_{5}(K_{6})$ such that $\Phi^{-1}_{K_{6},Q_{7}}(\gamma)=\left\{\gamma'_{1},\gamma'_{2}\right\}$. Thus we have $\tilde{\omega}(\gamma)=\omega(\gamma'_{1})+\omega(\gamma'_{2})=-1$. Finally, if $\gamma$ is an element in $\Gamma(Q_{7})\setminus \Gamma_{5}(Q_{7})\cup \Gamma_{6}(Q_{7})\cup \Gamma_{7}(Q_{7})$, we have $\tilde{\omega}(\gamma)=0$. In conclusion, we see that 
\begin{eqnarray*}
\tilde{\omega}(\gamma)=\left\{
       \begin{array}{@{\,}ll}
       1 & \mbox{if $\gamma\in \Gamma_{7}(Q_{7})\cup \left\{\delta\in \Gamma_{6}(Q_{7})\ |\ \delta\not\ni x\right\}$}\\
       -1 & \mbox{if $\gamma\in \left\{\delta\in \Gamma_{6}(Q_{7})\ |\ \delta\ni x,u,v,w\right\}\cup \Gamma_{5}(Q_{7})$}\\
       0 & \mbox{otherwise}
       \end{array}
     \right.
\end{eqnarray*}
for an element $\gamma$ in $\Gamma(Q_{7})$. 
\end{Example}

\begin{Example}\label{ex2}
Let $\tilde{\omega}$ be the map from $\bar{\Gamma}(H_{8})$ to ${\mathbb Z}$ as in (\ref{c2}). Then we see that 
\begin{eqnarray*}
\tilde{\omega}(\gamma)=\left\{
       \begin{array}{@{\,}ll}
       7 & \mbox{if $\gamma\in \Gamma_{8}(H_{8})\cup \left\{\delta\in \Gamma_{7}(H_{8})\ |\ \delta\not\ni x\right\}$}\\
       1 & \mbox{if $\gamma\in \left\{\delta\in \Gamma_{7}(H_{8})\ |\ \delta\ni x,\ \delta\not\supset \left\{u,v,w\right\}\right\}$}\\

       -6 & \mbox{if $\gamma\in \left\{\delta\in \Gamma_{7}(H_{8})\ |\ \delta\ni x,u,v,w\right\}\cup \left\{\delta\in \Gamma_{6}(H_{8})\ |\ \delta\not\ni x\right\}$}\\
       -8 & \mbox{if $\gamma\in \left\{\delta\in \Gamma_{6}(H_{8})\ |\ \delta \ni x,\ \delta\not\supset \left\{u,v,w\right\}\right\}$}\\
       -2 & \mbox{if $\gamma\in \left\{\delta\in \Gamma_{6}(H_{8})\ |\ \delta\ni x,u,v,w\right\}\cup \Gamma_{5}(H_{8})$}\\

       0 & \mbox{otherwise}
       \end{array}
     \right.
\end{eqnarray*}
for an element $\gamma$ in $\Gamma(H_{8})$ in a similar way as in Example \ref{ex1}, and also see that 
\begin{eqnarray*}
\tilde{\omega}(\gamma)=\left\{
       \begin{array}{@{\,}ll}
       1 & \mbox{if $\gamma\in \Gamma^{(2)}_{5,3}(H_{8})\cup \Gamma^{(2)}_{4,4}(H_{8})\cup \left\{\mu\in \Gamma^{(2)}_{4,3}(H_{8})\ |\ \mu\not\supset\left\{x,u,v,w\right\}\right\}$}\\       
       0 & \mbox{otherwise}
       \end{array}
     \right.
\end{eqnarray*}
for an element $\gamma$ in $\Gamma^{(2)}(Q_{7})$. Now we denote the set 
\begin{eqnarray*}
\Gamma_{8}(H_{8})\cup \left\{\delta\in \Gamma_{7}(H_{8})\ |\ \delta\not\ni x\right\}\cup \left\{\delta\in \Gamma_{7}(H_{8})\ |\ \delta\ni x,\ \delta\not\supset \left\{u,v,w\right\}\right\}
\end{eqnarray*}
by $\Gamma$. Then it follows that 
\begin{eqnarray*}
\sum_{\gamma\in \Gamma}{\rm Arf}(f(\gamma))\equiv 1\pmod{2}
\end{eqnarray*}
for any element $f$ in ${\rm SE}(H_{8})$. We remark here that the restricted map of $\Phi_{K_{7},H_{8}}$ on ${\Gamma_{7}(K_{7})}$ is a bijection from ${\Gamma_{7}(K_{7})}$ to $\Gamma$. On the other hand, the restricted map of $\Phi_{H_{8},F_{9}}\circ \Phi_{K_{7},H_{8}}$ on ${\Gamma_{7}(K_{7})}$ is not injective, see Figure \ref{twice2}.
\end{Example}

\begin{figure}[htbp]
      \begin{center}
\scalebox{0.45}{\includegraphics*{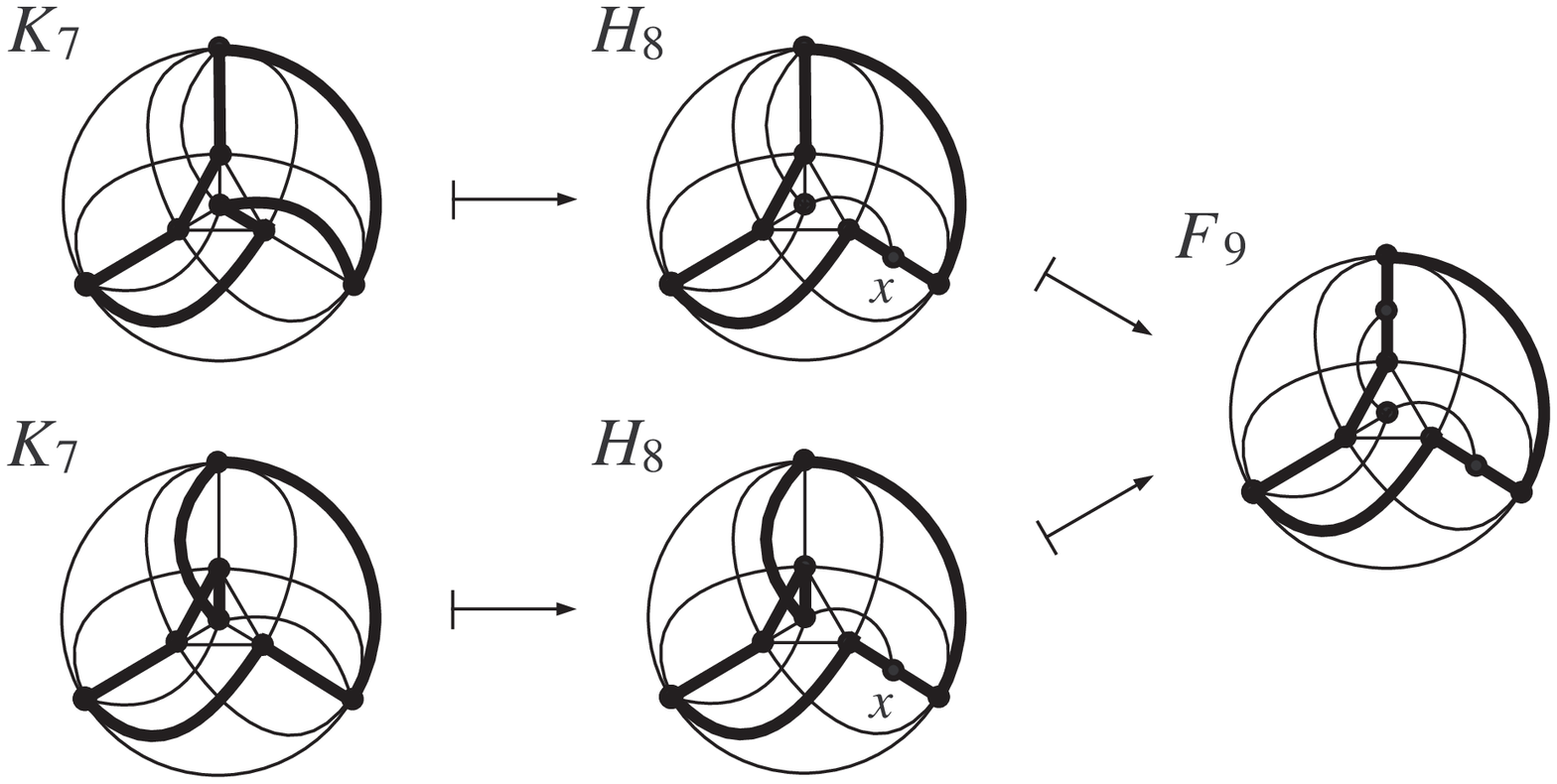}}
      \end{center}
   \caption{}
  \label{twice2}
\end{figure} 

\begin{Remark}\label{comments} 
{\rm

Let $G$ be a graph which is obtained from $K_{7}$ by a finite sequence of $\triangle Y$-exchanges. Then by Corollary \ref{main_cor} (2), for any element $f$ in ${\rm SE}(G)$ there exists an element $\gamma$ in $\Gamma(G)$ such that ${\rm Arf}(f(\gamma))=1$. This fact is also shown by applying Theorem \ref{CG} (2) and Proposition \ref{map} directly as follows. It is sufficient to show that if for any element $g$ in ${\rm SE}(G_{\triangle})$ there exists an element $\gamma'$ in $\Gamma(G_{\triangle})$ such that ${\rm Arf}(g(\gamma'))=1$, then for any element $f$ in ${\rm SE}(G_{Y})$ there exists an element $\gamma$ in $\Gamma(G_{Y})$ such that ${\rm Arf}(f(\gamma))=1$. Let $f$ be an element in ${\rm SE}(G_{Y})$. Then there exists an element $\gamma'$ in $\Gamma(G_{\triangle})$ such that ${\rm Arf}(\varphi(f)(\gamma'))=1$. Note that $\gamma'\neq \triangle$ because $\varphi(f)(\triangle)$ is a trivial knot. Let $\gamma$ be the image of $\gamma'$ by $\Phi$. Then we have ${\rm Arf}(f(\gamma))={\rm Arf}(f(\Phi(\gamma')))={\rm Arf}(\varphi(f)(\gamma'))=1$. Corollary \ref{main_cor} (2) insists on the result that is stronger than the fact above. Namely, there exists a subset $\Gamma$ of $\Gamma(G)$ which depends on only $G$ such that the sum of the Arf invariants over all of the images of the elements in $\Gamma$ by $f$ is odd for any element $f$ in ${\rm SE}(G)$. 
}
\end{Remark}


%
{\normalsize
}


\begin{thebibliography}{99}


\bibitem{CG83}
J. H. Conway and C. McA. Gordon, 
Knots and links in spatial graphs, 
{\it J. Graph Theory} {\bf 7} (1983), 445--453. 







\bibitem{HNTY10}
R. Hanaki, R. Nikkuni, K. Taniyama and A. Yamazaki, 
On intrinsically knotted or completely $3$-linked graphs, {\it Pacific J. Math.}, to appear. (arXiv:math.{\tt 1006.0698})




\bibitem{Ka83}
L. H. Kauffman, 
{\it Formal knot theory}, 
Mathematical Notes, {\bf 30}. {\it Princeton University Press, Princeton, NJ,} 1983. 



\bibitem{MRS88}
R. Motwani, A. Raghunathan and H. Saran, 
Constructive results from graph minors: Linkless embeddings, 
{\it 29th Annual Symposium on Foundations of Computer Science, IEEE,} 1988,
398--409.


\bibitem{N09b}
R. Nikkuni, 
A refinement of the Conway-Gordon theorems, 
{\it Topology Appl.} {\bf 156} (2009), 2782-2794. 

\bibitem{D10}
D. O'Donnol, 
Knotting and linking in the Petersen family, preprint. (arXiv:math.{\tt 1008.0377 })

\bibitem{Rober65}
R. A. Robertello, 
An invariant of knot cobordism, 
{\it Comm. Pure Appl. Math.} {\bf 18} (1965), 543--555. 


\bibitem{S84}
H. Sachs, 
On spatial representations of finite graphs, 
{\it Finite and infinite sets, Vol. I, II (Eger, 1981),} 649--662, 
Colloq. Math. Soc. Janos Bolyai, {\bf 37}, {\it North-Holland, Amsterdam,} 1984. 

\bibitem{TY01}
K. Taniyama and A. Yasuhara, 
Realization of knots and links in a spatial graph, 
{\it Topology Appl.} {\bf 112} (2001), 87--109.


\end{thebibliography}
\end{document}